\documentclass[11pt,reqno]{amsart}

\usepackage{hyperref}
\usepackage{graphicx}
\usepackage{color}
\usepackage{amsfonts,amssymb}
\usepackage{bbm} 

\usepackage{mathrsfs}
\usepackage{a4wide}

\newtheorem{neu}{}[section]

\newtheorem*{Cor*}{Corollary}

\newtheorem{Thm}[neu]{Theorem}
\newtheorem{Theorem}{Theorem}
\newtheorem*{Thm*}{Theorem}
\newtheorem{Prop}[neu]{Proposition}
\newtheorem*{Prop*}{Proposition}
\theoremstyle{definition}
\newtheorem{Lemma}[neu]{Lemma}
\newtheorem*{Rmk*}{Remark}
\newtheorem{Rmk}[neu]{Remark}

\newtheorem*{Ex*}{Example}

\newtheorem*{Qu*}{Question}

\newtheorem{Def}[neu]{Definition}

\newtheorem*{Conv*}{Convention}

\newcommand{\N}{\mathbb{N}}

\newcommand{\Z}{\mathbb{Z}}
\newcommand{\R}{\mathbb{R}}

\newcommand{\pf}{\longrightarrow}

\newcommand{\wrt}{with respect to }

\newcommand{\CZ}{\mu_{\mathrm{CZ}}}

\newcommand{\im}{\mathrm{im\,}}

\newcommand{\om}{\omega}

\newcommand{\ev}{\mathrm{ev}}

\newcommand{\A}{\mathcal{A}}

\newcommand{\M}{\mathcal{M}}

\newcommand{\B}{\mathcal{B}}
\newcommand{\E}{\mathcal{E}}
\renewcommand{\L}{\mathscr{L}}

\renewcommand{\H}{\mathrm{H}}

\newcommand{\Ham}{\mathrm{Ham}}

\newcommand{\RFH}{\mathrm{RFH}}
\newcommand{\RFC}{\mathrm{RFC}}

\newcommand{\Crit}{\mathrm{Crit}}

\newcommand{\beq}{\begin{equation}}
\newcommand{\beqn}{\begin{equation}\nonumber}
\newcommand{\eeq}{\end{equation}}

\newcommand{\bea}{\begin{equation}\begin{aligned}}
\newcommand{\bean}{\begin{equation}\begin{aligned}\nonumber}
\newcommand{\eea}{\end{aligned}\end{equation}}

\newcommand{\Mp}{\mathfrak{M}}

\numberwithin{equation}{section}

\definecolor{Urs}{rgb}{0,.7,0}
\definecolor{Peter}{rgb}{0,0,1}
\definecolor{red}{rgb}{1,0,0}

\newcommand{\He}{\mathscr{H}}

\begin{document}
\title{Infinitely many leaf-wise intersections on Cotangent bundles}
\author{Peter Albers}
\author{Urs Frauenfelder}
\address{
    Peter Albers\\
    Department of Mathematics\\
    Purdue University}
\email{palbers@math.purdue.edu}
\address{
    Urs Frauenfelder\\
    Department of Mathematics and Research Institute of Mathematics\\
    Seoul National University}
\email{frauenf@snu.ac.kr}
\begin{abstract}
If the homology of the free loop space of a closed manifold $B$ is infinite dimensional then generically there exist infinitely many leaf-wise intersection points for fiber-wise star-shaped hypersurfaces in $T^*B$.
\end{abstract}
\maketitle

\section{Introduction}

Let $B$ be a closed manifold and $\Sigma\subset T^*B$ be a fiber-wise star-shaped hypersurface \wrt the standard Liouville vector field. $\Sigma$ is foliated by the Reeb flow associated to the Liouville 1-form $\lambda$. We denote by $L_x$ the leaf through $x\in\Sigma$. Let $\psi\in\Ham_c(T^*B)$ be in the space of Hamiltonian diffeomorphisms generated by compactly supported time dependent Hamiltonian functions. Then a leaf-wise intersection is a point $x\in\Sigma$ with the property $\psi(x)\in L_x$. The search for leaf-wise intersections was initiated by Moser in \cite{Moser_A_fixed_point_theorem_in_symplectic_geometry} and pursued further in 
\cite{Banyaga_On_fixed_points_of_symplectic_maps,Hofer_On_the_topological_properties_of_symplectic_maps,
Ekeland_Hofer_Two_symplectic_fixed_point_theorems_with_applications_to_Hamiltonian_dynamics,Ginzburg_Coisotropic_intersections,Dragnev_Symplectic_rigidity_symplectic_fixed_points_and_global_perturbations_of_Hamiltonian_systems,Albers_Frauenfelder_Leafwise_intersections_and_RFH,Ziltener_coisotropic,Gurel_leafwise_coisotropic_intersection,Kang_Existence_of_leafwise_intersection_points_in_the_unrestricted_case}. A brief history of the search for leaf-wise intersections is given below.

We call $\Sigma$ non-degenerate if Reeb orbits on $\Sigma$ form a  discrete set. A generic $\Sigma$ is non-degenerate, see \cite[Theorem B.1]{Cieliebak_Frauenfelder_Restrictions_to_displaceable_exact_contact_embeddings}. We denote by $\L_B$ the free loop space of $B$.

\begin{Theorem}\label{thm:generically_infinitely_many_LI}
Let $\dim\H_*(\L_B)=\infty$. If $\dim B\geq2$ and $\Sigma$ is non-degenerate then for a generic $\psi\in\Ham_c(T^*B)$ there exist infinitely many leaf-wise intersections.
\end{Theorem}

\begin{Rmk}$ $
\begin{itemize}
\item To our knowledge all so far known existence results for leaf-wise intersections assert only finite lower bounds. Moreover, all known results make smallness assumptions on either the $C^1$ or Hofer norm of $\psi$.\\[-1.5ex]
\item The assumption $\dim B\geq2$ is necessary as the example $B=S^1$ shows.\\[-1.5ex]
\item If $\pi_1(B)$ is finite then $\dim\H_*(\L_B)=\infty$ by a theorem of Vigu\'e-Poirrier and Sullivan \cite{Sullivan_Vigue_the_homology_theory_of_the_closed_geodesic_problem}. If the number of conjugacy classes of $\pi_1(B)$ is infinite then $\dim\H_0(\L_B)=\infty$. Therefore, the only remaining case is if $\pi_1(B)$ is infinite but the number of conjugacy classes of $\pi_1(B)$ is finite.
\end{itemize}
\end{Rmk}

\subsection{History of the problem and related results}

The problem addressed above is a special case of the leaf-wise coisotropic intersection problem. For that let $N\subset(M,\om)$ be a coisotropic submanifold. Then $N$ is foliated by isotropic leafs, see \cite[Section 3.3]{McDuff_Salamon_introduction_symplectic_topology}. The problem asks for a leaf $L$ such that $\phi(L)\cap L\neq\emptyset$ for $\phi\in\Ham_c(M,\om)$. 

The first existence result was obtained by Moser in \cite{Moser_A_fixed_point_theorem_in_symplectic_geometry} for simply connected $M$ and $C^1$-small $\phi$. This was later generalized by Banyaga \cite{Banyaga_On_fixed_points_of_symplectic_maps} to non-simply connected $M$.

The $C^1$-smallness assumption was replaced by Hofer, Ekeland-Hofer in \cite{Hofer_On_the_topological_properties_of_symplectic_maps},\cite{Ekeland_Hofer_Two_symplectic_fixed_point_theorems_with_applications_to_Hamiltonian_dynamics} for hypersurfaces of restricted contact type in $\R^{2n}$ by a much weaker smallness assumption, namely that the Hofer norm of $\phi$ is smaller than a certain symplectic capacity. Only recently, the result by Ekeland-Hofer was generalized in two different directions. It was extended by Dragnev \cite{Dragnev_Symplectic_rigidity_symplectic_fixed_points_and_global_perturbations_of_Hamiltonian_systems} to so-called ``coisotropic submanifolds of contact type in $\R^{2n}$''. Ginzburg \cite{Ginzburg_Coisotropic_intersections} generalized from restricted contact type in $\R^{2n}$ to restricted contact type in subcritical Stein manifolds. Moreover, examples by Ginzburg \cite{Ginzburg_Coisotropic_intersections} show that the Ekeland-Hofer result is a symplectic rigidity result, namely it becomes wrong for arbitrary hypersurfaces. In \cite{Albers_Frauenfelder_Leafwise_intersections_and_RFH} the authors proved multiplicity results for restricted contact-type hypersurfaces. These were recently generalized by Kang in \cite{Kang_Existence_of_leafwise_intersection_points_in_the_unrestricted_case}. Ziltener \cite{Ziltener_coisotropic} established multiplicity results in the special case of fibrations. Finally, Gurel \cite{Gurel_leafwise_coisotropic_intersection} obtained existence results for leaf-wise intersections for coisotropic submanifolds of restricted contact type.

\subsubsection*{Acknowledgments}
The authors are partially supported by the German Research Foundation (DFG) through Priority Program 1154 "Global Differential Geometry", grant FR 2637/1-1, and NSF grant DMS-0903856. 

\section{Leaf-wise intersections  and Rabinowitz Floer homology}

Let $(M,\om)$ be a symplectic manifold and $f\in C^\infty(M)$ an autonomous Hamiltonian function. Since energy is preserved the hypersurface $\Sigma:=f^{-1}(0)$ is invariant under the Hamiltonian flow $\phi_f^t$ of $f$. The Hamiltonian flow $\phi_f^t$ is generated by the Hamiltonian vector filed $X_f$ which is uniquely defined by the equation $\om(X_f,\cdot)=df$. If $0$ is a regular value of $f$ the hypersurface is a coisotropic submanifold which is foliated by 1-dimensional isotropic leaves, see \cite[Section 3.3]{McDuff_Salamon_introduction_symplectic_topology}. If we denote by $L_x$ the leaf through $x\in\Sigma$ we have the equality
\beq
L_x=\bigcup_{t\in\R}\phi_f^t(x)\;.
\eeq
Given a time-dependent Hamiltonian function $H:[0,1]\times M\pf\R$ with Hamiltonian flow $\phi_H^t$ we are interested in points $x\in\Sigma$ with the property
\beq
\phi_H^1(x)\in L_x\;.
\eeq
This notion was introduced and studied by Moser in \cite{Moser_A_fixed_point_theorem_in_symplectic_geometry}. Such points are called leaf-wise intersections. For a physical interpretation of leaf-wise intersections it is useful to think of the Hamiltonian $H$ as a perturbation of the conservative Hamiltonian system $\phi_f^t$. More dramatically one can think of $H$ as an earthquake lasting from time $t=0$ to $t=1$. Without the earthquake the physical system propagates along a fixed leaf of $\Sigma$. Now we can ask whether the physical system survives the earthquake unharmed. This happens precisely if there exists a leaf-wise intersection. We refer to the article \cite{Moser_A_fixed_point_theorem_in_symplectic_geometry} by Moser for further physical applications and examples.

\begin{Def}\label{def:periodic_LI}
A leaf-wise intersection $x\in\Sigma$ is called periodic if the leaf $L_x$ is a closed orbit of the flow $\phi_f^t$.
\end{Def}

\begin{Def}\label{def:Moser_pair}
A pair $\Mp=(F,H)$ of Hamiltonian functions $F,H:S^1\times M\pf R$ is called a Moser pair if it satisfies
\beq
F(t,\cdot)=0\quad\forall t\in[\tfrac12,1]\qquad\text{and}\qquad H(t,\cdot)=0\quad\forall t\in[0,\tfrac12]\;,
\eeq
and $F$ is of the form $F(t,x)=\rho(t)f(x)$ for some smooth map $\rho:S^1\to S^1$ with $\int_0^1\rho(t) dt=1$ and $f:M\pf\R$. 
\end{Def}

\begin{Def}\label{def:set_of_half_constant_Hamiltonians}
We set
\beq
\mathcal{H}:=\{H\in C^\infty(S^1\times M)\mid H\text{ has compact support and } H(t,\cdot)=0\quad\forall t\in[0,\tfrac12]\}
\eeq
\end{Def}

\begin{Rmk}
It's easy to see that the $\Ham(M,\om)\equiv\{\phi_H^1\mid H\in\mathcal{H}\}$, e.g.\cite{Albers_Frauenfelder_Leafwise_intersections_and_RFH}.
\end{Rmk}

Let $(M,\om=-d\lambda)$ be an exact symplectic manifold. Then for a Moser pair $\Mp=(F,H)$ the perturbed Rabinowitz action functional is defined by
\bea
\A^\Mp:\L_M\times\R&\pf\R\\
(v,\eta)&\mapsto\int_{S^1}v^*\lambda-\int_0^1H(t,v)dt-\eta\int_0^1F(t,v)dt
\eea
where $\L_M:=C^\infty(S^1,M)$. We recall that $\om(X_F,\cdot)=dF(\cdot)$. Then a critical point $(v,\eta)$ of $\A^\Mp$ is a solution of 
\beq\label{eqn:critical_points_eqn}
\left. 
\begin{aligned}
\partial_tv=\eta X_F(t,v)+X_H(t,v)\\
\int_0^1F(t,v)dt=0
\end{aligned}\right\}
\eeq
We observed in \cite{Albers_Frauenfelder_Leafwise_intersections_and_RFH} that critical points of $\A^\Mp$ give rise to leaf-wise intersections. 

\begin{Prop}[\cite{Albers_Frauenfelder_Leafwise_intersections_and_RFH}]\label{prop:critical_points_give_LI}
Let $(v,\eta)$ be a critical point of $\A^\Mp$ then $x:=v(\tfrac12)\in f^{-1}(0)$ and
\beq
\phi_H^1(x)\in L_x
\eeq
thus, $x$ is a leaf-wise intersection.\\[1ex]
Moreover, the map $\Crit\A^\Mp\to\{\text{leaf-wise intersections}\}$ is injective unless there exists a periodic leaf-wise intersection (see Definition \ref{def:periodic_LI}).
\end{Prop}

\begin{Def}
A Moser pair $\Mp=(F,H)$ is of contact-type if the following four conditions hold.
\begin{enumerate}
\item $0$ is a regular value of $f$.
\item $df$ has compact support.
\item The hypersurface $f^{-1}(0)$ is a closed restricted contact type hypersurface of $(M,\lambda)$.
\item The Hamiltonian vector field $X_f$ restricts to the Reeb vector field on $f^{-1}(0)$.
\end{enumerate}
\end{Def}

\begin{Rmk}
If $\Sigma\subset T^*B$ is a fiber-wise star-shaped hypersurface there exists a contact-type Moser pair $\Mp$ with $\Sigma=f^{-1}(0)$.
\end{Rmk}

\begin{Def}
A Moser pair $\Mp$ is called regular if $\A^\Mp$ is Morse.
\end{Def}

We recall the following 
\begin{Prop}[\cite{Albers_Frauenfelder_Leafwise_intersections_and_RFH}]\label{prop:generic_is_regular} 
A generic contact-type Moser pair is regular. 
\end{Prop}

For a regular contact-type Moser pair $\Mp$ on an exact symplectic manifold which is convex at infinity Rabinowitz Floer homology $\RFH_*(\Mp)$ is defined from the chain complex
\beq
\RFC_k(\Mp):=\Big\{\xi=\!\!\!\!\!\!\!\sum_{\CZ(c)=k}\!\!\!\!\xi_c\,c\mid\#\{c\in\Crit{\A^\Mp}\mid\xi_c\neq0\in\Z/2,\A^\Mp(c)\geq\kappa\}<\infty\;\forall\kappa\in\R\Big\}
\eeq
where the boundary operator is defined by counting gradient flow lines of $\A^\Mp$ in the sense of Floer homology, see \cite{Cieliebak_Frauenfelder_Restrictions_to_displaceable_exact_contact_embeddings,Albers_Frauenfelder_Leafwise_intersections_and_RFH} for details. In particular, on cotangent bundles $T^*B$ $\RFH_*(\Mp)$ is well-defined.

If the Moser pair is of the form $\Mp=(F,0)$ then $\A^\Mp$ is never Morse. But for a generic $F$ the action functional  $\A^\Mp$ is Morse-Bott with critical manifold being the disjoint union of constant solutions of the form $(p,0)$, $p\in f^{-1}(0)$, and a family of circles corresponding to closed characteristics of $\om$ on $f^{-1}(0)$.

\begin{Def}
A Moser pair is called weakly regular if it is of the form just described or if it is regular. 
\end{Def}

\begin{Rmk}\label{rmk:perturb_by_Morse_fctn}
For weakly regular Moser pairs $\Mp$ Rabinowitz Floer homology $\RFH_*(\Mp)$ can still be defined by taking the critical points of a Morse function on the critical manifolds as generators, see \cite{Cieliebak_Frauenfelder_Restrictions_to_displaceable_exact_contact_embeddings} for details. 
\end{Rmk}

\begin{Rmk}\label{rmk:fiberwise_starshaped_are_all_isotopic}
We note that if we have two Moser pairs $\Mp_0=(F_0,H_0)$ and $\Mp_1=(F_1,H_1)$ associated to two fiber-wise star-shaped hypersurfaces $\Sigma_0$ and $\Sigma_1$ then they can be joint through a smooth family of Moser pairs $\Mp^r=(F^r,H^r)$ such that the corresponding hypersurfaces $\Sigma_r$ remain fiber-wise star-shaped. In particular, each $\Mp^r$ is a contact-type Moser pair.
\end{Rmk}

Let $\Mp^r=(F^r,H^r)$, $r\in[0,1]$ be a smooth family of contact-type Moser pairs. We fix once for all a smooth function $\beta\in C^\infty(\R,[0,1])$ satisfying $\beta(s)=0$ for $s\leq0$, $\beta(s)=1$ for $s\geq1$, and $0\leq\beta'\leq2$. Then we set
\beq
F_s:=F^{\beta(s)},\;H_s:=H^{\beta(s)},\;\text{and}\; \Mp_s:=(F_s,H_s)
\eeq
for $s\in\R$. The corresponding $s$-dependent Rabinowitz action functional is
\beq\label{eqn:Rabinowitz action functional}
\A_s(v,\eta):=\int_{S^1}v^*\lambda-\int_0^1H_s(t,v(t))dt-\eta\int_0^1F_s(t,v(t))dt
\eeq
It is used to define the standard continuation homomorphisms in Rabinowitz Floer homology, that is, given two weakly regular Moser pairs $\Mp^0$ and $\Mp^1$ there exist natural isomorphisms
\beq
m^{\Mp^0}_{\Mp^1}:\RFH_*(\Mp^0)\pf\RFH_*(\Mp^1)\,,
\eeq
see \cite{Albers_Frauenfelder_Leafwise_intersections_and_RFH} for details.

\section{Proof of Theorem \ref{thm:generically_infinitely_many_LI}}

Let $(B,g)$ be a closed Riemannian manifold and $S^*_gB$ the unit cotangent bundle with respect to $g$. Cutting off the function $\frac12(||p||_g^2-1)$ outside a large compact subset of $T^*B$ gives rise to a contact-type Moser pair $\Mp_0=(F_0,0)$ for $S^*_gB$.

\begin{Rmk}\label{rmk:Abraham}
According to a Theorem by Abraham \cite{Abraham_Bumpy_metrics} for a generic metric $g$ the Moser pair $\Mp_0=(F_0,0)$ is weakly regular. More precisely, every bumpy metric satisfies this condition.
\end{Rmk}

We recall 
\begin{Thm}\cite{Cieliebak_Frauenfelder_Oancea_Rabinowitz_Floer_homology_and_symplectic_homology,Abbondandolo_Schwarz_Estimates_and_computations_in_Rabinowitz_Floer_homology}
For degrees $*\neq0,1$
\beq
\RFH_*(\Mp_0)\cong\begin{cases}
H_*(\L_B)\\[.5ex]
H^{-*+1}(\L_B)
\end{cases}
\eeq
\end{Thm}

\begin{proof}[Proof of Theorem \ref{thm:generically_infinitely_many_LI}]
We fix a fiber-wise star-shaped hypersurface $\Sigma$ and $\psi\in\Ham_c(T^*B)$. This gives rise to a Moser pair $\Mp=(F,H)$. If $\Sigma$ is non-degenerate and $\psi$ is generic the perturbed Rabinowitz action functional $\A^\Mp$ is Morse, see Proposition \ref{prop:generic_is_regular}. Since $\Sigma$ is fiber-wise star-shaped the Moser pair $\Mp$ can be joined to $\Mp_0$ through contact-type Moser pairs, see Remark \ref{rmk:fiberwise_starshaped_are_all_isotopic}. Thus, using the continuation isomorphism
\beq
m^{\Mp_0}_{\Mp}:\RFH_*(\Mp_0)\pf\RFH_*(\Mp)
\eeq
we conclude that 
\beq
\RFH_*(\Mp)\cong\begin{cases}
H_*(\L_B)\\[.5ex]
H^{-*+1}(\L_B)
\end{cases}
\eeq
Since we assume that $\dim\H_*(\L_B)=\infty$ we have $\dim\RFH_*(\Mp)=\infty$ and therefore, the Morse function $\A^\Mp$ has infinitely many critical points. Now, Proposition \ref{prop:critical_points_give_LI} implies that there exist infinitely many leaf-wise intersections or a period leaf-wise intersection. Thus, to prove Theorem \ref{thm:generically_infinitely_many_LI} we need to exclude the latter for a generic $\psi\in\Ham_c(T^*B)$. That is, we need to make sure that for generic $\psi$ the critical points of $\A^\Mp$ do not intersect closed Reeb orbits. This is exactly the content of Theorem \ref{thm:evaluation}.
\end{proof}

We recall that a fiber-wise star-shaped hypersurface $\Sigma$ is called non-degenerate if the set $\mathcal{R}$ of Reeb orbits on $\Sigma$ form a discrete set. A generic $\Sigma$ is non-degenerate, see \cite[Theorem B.1]{Cieliebak_Frauenfelder_Restrictions_to_displaceable_exact_contact_embeddings}. 

\begin{Thm}\label{thm:evaluation}
Let $\Sigma=f^{-1}(0)\subset T^*B$ be a non-degenerate star-shaped hypersurface and $\Mp_0=(F_0,0)$ be the corresponding weakly regular Moser pair. If $\dim B\geq2$ then the set
\beq
\mathcal{H}_\Sigma:=\{H\in\mathcal{H}\mid \A^{(F_0,H)}\text{ is Morse and }\im(x)\cap\im(y)=\emptyset\quad\forall x\in\Crit\A^{(F_0,H)},y\in\mathcal{R}\}
\eeq
is generic in $\mathcal{H}$ (see Definition \ref{def:set_of_half_constant_Hamiltonians}).
\end{Thm}

\begin{proof}
We set $M:=T^*B$, $\L=W^{1,2}(S^1,M)$, and $\mathcal{H}^k:=\{H\in C^k(S^1\times M)\mid H(t,\cdot)=0\quad\forall t\in[0,\tfrac12]\}$. Furthermore, we define the Banach space bundle $\E\pf\L$ by $\E_v=L^2(S^1,v^*TM)$. We consider the section $S:\L\times\R\times\mathcal{H}^k\pf\E^\vee\times\R$ defined by
\beq
S(v,\eta,H):=d\A^{(F_0,H)}(v,\eta)\;.
\eeq 
Its vertical differential $DS:T_{(v_0,\eta_0,H)}\L\times\R\times\mathcal{H}^k\pf\E_{(v_0,\eta_0,H)}^\vee$ at $(v_0,\eta_0,H)\in S^{-1}(0)$ is
\beq
DS_{(v_0,\eta_0,H)}[(\hat{v},\hat{\eta},\hat{H})]=\He_{\A^{(F_0,H)}}(v_0,\eta_0)\big[(\hat{v},\hat{\eta},\hat{H})\,;\:\bullet\:\big]+\int_0^1\hat{H}(t,v_0)dt
\eeq
where $\He_{\A^{(F_0,H)}}$ is the Hessian of $\A^{(F_0,H)}$. In \cite{Albers_Frauenfelder_Leafwise_intersections_and_RFH} we proved the following.
\begin{Prop}\label{prop:linearized_operator_surjective}
The operator $DS_{(v_0,\eta_0,H)}$ is surjective for $(v_0,\eta_0,H)\in S^{-1}(0)$. In fact, $DS_{(v_0,\eta_0,H)}$ is surjective when restricted to the space
\beq
\mathcal{V}:=\{(\hat{v},\hat{\eta},\hat{H})\in T_{(v_0,\eta_0,H)}\L\times\R\times\mathcal{H}^k\mid \hat{v}(\tfrac12)=0\}\;.
\eeq
\end{Prop}

Thus, by the implicit function theorem the universal moduli space 
\beq
\M:=S^{-1}(0)
\eeq
is a smooth Banach manifold. We consider the projection $\Pi:\M\pf\mathcal{H}^k$. Then $\A^{(F_0,H)}$ is Morse if and only if $H$ is a regular value of $\Pi$, which by the theorem of Sard-Smale form a generic set (for $k$ large enough). Moreover, the Morse condition is $C^k$-open. Thus, for functions in an open and dense subset of $\mathcal{H}^k$ the functional $\A^{(F_0,H)}$ is Morse. 

Next we define the evaluation map 
\bea
\ev:\M&\pf\Sigma\\
(v_0,\eta_0,H)&\mapsto v_0(\tfrac12)
\eea
From Proposition \ref{prop:linearized_operator_surjective} together with Lemma \ref{lemma:Dietmar} below it follows that the evaluation map $\ev_H:=\ev(\cdot,\cdot,H):\Crit\A^{(F_0,H)}\pf\Sigma$ is a submersion for a generic choice of $H$.  Thus, the preimage of the one dimensional set $\mathcal{R}^\tau:=\{\text{Reeb orbits with period }\leq\tau\}$ under $\ev_H$ doesn't intersect $\Crit\A^{(F_0,H)}$ using that $\dim T^*B\geq4$. Therefore, the set
\beq
\mathcal{H}_\Sigma^n:=\{H\in\mathcal{H}^n\mid \A^{(F_0,H)}\text{ is Morse and }\im(x)\cap\im(y)=\emptyset\quad\forall x\in\Crit\A^{(F_0,H)},y\in\mathcal{R}^n\}
\eeq
is generic in $\mathcal{H}$ for all $n\in\N$. Now, the set $\mathcal{H}_\Sigma$ is a countable intersection of the sets $\mathcal{H}_\Sigma^n$, $n\in\N$. This proves the assertion of Theorem \ref{thm:evaluation}.
\end{proof}

We learned the following Lemma from Dietmar Salamon. 
\begin{Lemma}\label{lemma:Dietmar}
Let $\E\pf\B$ be a Banach bundle and $s:\B\pf\E$ a smooth section. Moreover, let $\phi:\B\pf N$ be a smooth map into the Banach manifold $N$.
We fix a point $x\in s^{-1}(0)\subset\B$ and set $K:=\ker d\phi(x)\subset T_x\B$ and assume the following two conditions.
\begin{enumerate}
\item The vertical differential $Ds|_K:K\pf\E_x$ is surjective.
\item $d\phi(x):T_x\B\pf T_{\phi(x)}N$ is surjective.
\end{enumerate}
Then  $d\phi(x)|_{\ker Ds(x)}:\ker Ds(x)\pf T_{\phi(x)}N$ is surjective.
\end{Lemma}

For convenience we provide a proof here.

\begin{proof}
We fix $\xi\in T_{\phi(x)}N$. Condition (2) implies that there exists $\eta\in T_x\B$ satisfying $d\phi(x)\eta=\xi$. Condition (1)
implies that there exists $\zeta\in K\subset T_x\B$ satisfying $Ds(x)\zeta=Ds(x)\eta$. We set $\tau:=\eta-\zeta$ and compute
\beq
Ds(x)\tau=Ds(x)\eta-Ds(x)\zeta=0
\eeq
thus, $\tau\in\ker Ds(x)$. Moreover,
\beq
d\phi(x)\tau=d\phi(x)\eta-\underbrace{d\phi(x)\zeta}_{=0}=d\phi(x)\eta=\xi
\eeq
proving the Lemma.
\end{proof}

\bibliographystyle{amsalpha}
\bibliography{../../../Bibtex/bibtex_paper_list}

\providecommand{\bysame}{\leavevmode\hbox to3em{\hrulefill}\thinspace}
\providecommand{\MR}{\relax\ifhmode\unskip\space\fi MR }
\providecommand{\MRhref}[2]{%
  \href{http://www.ams.org/mathscinet-getitem?mr=#1}{#2}
}
\providecommand{\href}[2]{#2}
\begin{thebibliography}{CFO09}

\bibitem[Abr70]{Abraham_Bumpy_metrics}
R.~Abraham, \emph{Bumpy metrics}, Global {A}nalysis ({P}roc. {S}ympos. {P}ure
  {M}ath., {V}ol. {XIV}, {B}erkeley, {C}alif., 1968), Amer. Math. Soc.,
  Providence, R.I., 1970, pp.~1--3.

\bibitem[AF08]{Albers_Frauenfelder_Leafwise_intersections_and_RFH}
P.~Albers and U.~Frauenfelder, \emph{{Leaf-wise intersections and Rabinowitz
  Floer homology}}, 2008, arXiv:0810.3845.

\bibitem[AS09]{Abbondandolo_Schwarz_Estimates_and_computations_in_Rabinowitz_F%
loer_homology}
A.~Abbondandolo and M.~Schwarz, \emph{{Estimates and computations in
  Rabinowitz-Floer homology}}, 2009, arXiv:0907.1976.

\bibitem[Ban80]{Banyaga_On_fixed_points_of_symplectic_maps}
A.~Banyaga, \emph{On fixed points of symplectic maps}, Invent. Math.
  \textbf{56} (1980), no.~3, 215--229.

\bibitem[CF09]{Cieliebak_Frauenfelder_Restrictions_to_displaceable_exact_conta%
ct_embeddings}
K.~Cieliebak and U.~Frauenfelder, \emph{{A Floer homology for exact contact
  embeddings}}, Pacific J. Math. \textbf{293} (2009), no.~2, 251--316.

\bibitem[CFO09]{Cieliebak_Frauenfelder_Oancea_Rabinowitz_Floer_homology_and_sy%
mplectic_homology}
K.~Cieliebak, U.~Frauenfelder, and A.~Oancea, \emph{{Rabinowitz Floer homology
  and symplectic homology}}, 2009, arXiv:0903.0768, to appear in Annales
  Scientifiques de L'ENS.

\bibitem[Dra08]{Dragnev_Symplectic_rigidity_symplectic_fixed_points_and_global%
_perturbations_of_Hamiltonian_systems}
D.~L. Dragnev, \emph{Symplectic rigidity, symplectic fixed points, and global
  perturbations of {H}amiltonian systems}, Comm. Pure Appl. Math. \textbf{61}
  (2008), no.~3, 346--370.

\bibitem[EH89]{Ekeland_Hofer_Two_symplectic_fixed_point_theorems_with_applicat%
ions_to_Hamiltonian_dynamics}
I.~Ekeland and H.~Hofer, \emph{Two symplectic fixed-point theorems with
  applications to {H}amiltonian dynamics}, J. Math. Pures Appl. (9) \textbf{68}
  (1989), no.~4, 467--489 (1990).

\bibitem[Gin07]{Ginzburg_Coisotropic_intersections}
V.~L. Ginzburg, \emph{Coisotropic intersections}, Duke Math. J. \textbf{140}
  (2007), no.~1, 111--163.

\bibitem[Gur09]{Gurel_leafwise_coisotropic_intersection}
B.~Gurel, \emph{{Leafwise Coisotropic Intersections}}, 2009, arXiv:0905.4139,
  to appear in IMRN.

\bibitem[Hof90]{Hofer_On_the_topological_properties_of_symplectic_maps}
H.~Hofer, \emph{On the topological properties of symplectic maps}, Proc. Roy.
  Soc. Edinburgh Sect. A \textbf{115} (1990), no.~1-2, 25--38.

\bibitem[Kan09]{Kang_Existence_of_leafwise_intersection_points_in_the_unrestri%
cted_case}
J.~Kang, \emph{{Existence of leafwise intersection points in the unrestricted
  case}}, 2009, arXiv:0910.2369.

\bibitem[Mos78]{Moser_A_fixed_point_theorem_in_symplectic_geometry}
J.~Moser, \emph{A fixed point theorem in symplectic geometry}, Acta Math.
  \textbf{141} (1978), no.~1--2, 17--34.

\bibitem[MS98]{McDuff_Salamon_introduction_symplectic_topology}
D.~McDuff and D.~Salamon, \emph{Introduction to symplectic topology}, second
  ed., Oxford Mathematical Monographs, The Clarendon Press Oxford University
  Press, New York, 1998.

\bibitem[VPS76]{Sullivan_Vigue_the_homology_theory_of_the_closed_geodesic_prob%
lem}
M.~Vigu\'e-Poirrier and D.~Sullivan, \emph{{The homology theory of the closed
  geodesic problem}}, J. Differential Geometry \textbf{11} (1976), no.~4,
  633--644.

\bibitem[Zil08]{Ziltener_coisotropic}
F.~Ziltener, \emph{{Coisotropic Submanifolds, Leafwise Fixed Points, and
  Presymplectic Embeddings}}, 2008, arXiv:0811.3715, to appear in Journal of
  Symplectic Geometry.

\end{thebibliography}
\end{document}